\documentclass[12pt,a4]{article}
\usepackage{amsmath, amsthm, amssymb}
\usepackage[all]{xy}
\usepackage{amscd}
\newtheorem{theorem}{Theorem}[section]

\newtheorem{corollary}[theorem]{Corollary}

\newtheorem{lemma}[theorem]{Lemma}
\newtheorem{conjecture}{Conjecture}

\newenvironment{proof*}{\begin{trivlist}\item[\hskip\labelsep{\bf Proof of Theorem \ref{t2.4}.}]}{%
\hfill$\square$\rm\end{trivlist}}
\newenvironment{ack}{\begin{trivlist}\item[\hskip\labelsep{\bf Acknowledgements.}]}{%
\hfill$\square$\rm\end{trivlist}}

\everymath{\displaystyle}

\begin{document}
\bibliographystyle{plainnat}
\pagestyle{plain} \pagenumbering{arabic}
\date{}
\newpage
\title{\textbf{\small{A Note On The Non-Artinianness of Top Local Cohomology Modules }}}
\author{\small{Tu\u{g}ba Y\i{ı}ld\i{ı}r\i{ı}m}}
\maketitle
\begin{center}
\small{{Department of Mathematics} \\
{Istanbul Technical University}\\
{Maslak, 34469, Istanbul, Turkey}\\
E-mail: tugbayildirim@itu.edu.tr}
\end{center}
\maketitle

\begin{abstract}
\small
{Let $R$ be a Noetherian ring, $I$ an ideal of $R$ and $M$ an $R$-module. In this article, we examine the question of whether an arbitrary top local cohomology module, $\operatorname{H}^{\operatorname{cd}(I,M)}_I(M)$, is Artinian, or not. Several results related to this question are obtained; in particular, we prove that over a  Noetherian local unique factorization domain $R$ of dimension three, for a  finitely generated faithful module $M$, a top local cohomology module is Artinian if and only if $\operatorname{cd}(I,M)= 3$.}
\end{abstract}
 {\bf \emph{Keywords}:}\emph{ Local cohomology; Artinian module; Support.\\
 2000 Mathematics Subject Classification.  	13D45, 13E10.}
\section{Introduction}\label{intro} 
Local cohomology theory is an active area of research in commutative algebra, algebraic geometry and related fields. One of the classical and long-standing problem in this theory is to determine whether a given local cohomology module is Artinian, or not [Third Problem, \cite{Huneke}].
Regarding this problem, Erdo\u{g}du and Y\i ld\i r\i m \cite{Erdogdu-Yildirim} recently proved that  over a Noetherian local ring $(R,\mathfrak m)$, for a finitely generated module $M$ of dimension $n$, the top local cohomology module, $\operatorname{H}^{\operatorname{cd}(I,M)}_I(M)$, is non-Artinian for all ideals $I$ of $R$ with $d=\dim(M/IM)>0$ and ${\operatorname{cd}(I,M)}=n-d.$ Motivated by this result, we ask whether a top local cohomology module $\operatorname{H}^{\operatorname{cd}(I,M)}_I(M)$ is non-Artinian in the case when ${\operatorname{cd}(I,M)}=n-d+1$ and, in fact, we conjecture that, except in the trivial case $d=1,$ it always is: 

\begin{conjecture} \label{conj1} 
Let $(R,\mathfrak m)$ be a Noetherian local ring,  $M$ a finitely generated $R$-module of dimension $n$ and $I$ an ideal of $R$ with $\dim(M/IM)=d>1$. Then top local cohomology module, $\operatorname{H}^{\operatorname{cd}(I,M)}_I(M)$, is non-Artinian for all ideals $I$ of $R$  with $\operatorname{cd}(I,M)=n-d+1$.
\end{conjecture}
In this article, we first show that Conjecture \ref{conj1} holds true in the case when $R$ is a Noetherian local unique factorization domain of dimension three and $M$ a finitely generated faithful $R$-module. We then use this result to conclude that, under the same conditions on $R$ and $M$, the top local cohomology module, $\operatorname{H}^{\operatorname{cd}(I,M)}_I(M)$, is Artinian if and only if $\operatorname{cd}(I,M)= 3$.\par
It is a well-known fact that when $R$ is of dimension three, for a finitely generated $R$-module $M$, $\operatorname{Supp}(\operatorname{H}^{i}_I(M))$ is Zariski closed for all $i$  and for all ideals $I$ of $R$ [Corollary 2.6,  \cite{Marley}]. This fact together with our results above compelled  us to ask whether the assumptions of Conjecture \ref{conj1}, together with the condition $\operatorname{Supp}(\operatorname{H}^{i}_I(M))$ being finite for all $i$, implies that all top local cohomology modules are non-Artinian except in the trivial case when  $\operatorname{cd}(I,M)= \dim M$?
In this paper, we give an affirmative answer to that question and use this to provide conditions under which all top local cohomology modules, but the ones with $\operatorname{cd}(I,M)= \dim M$, are non-Artinian when the underlying ring $R$ is of dimension four. 
\section{Preliminaries}
The purpose of this section is to state and prove some basic results which we need them in the next section:\\
Let $R$ be a Noetherian ring, $I$ an ideal of $R$ and $M$ an $R$-module. The cohomological dimension of $M$ with respect to $I$, denoted by ${\operatorname{cd}(I,M)}$, is defined to be the supremum of the set of integers $i$ such that $\operatorname{H}^{i}_I(M)\neq 0.$\\
We begin this section with the following well-known results which are recently proved in \cite{Erdogdu-Yildirim}:
\begin{lemma}\label{l2.3}\emph{({\cite{Erdogdu-Yildirim}, Corollary 3.2})} Let $(R,\mathfrak m)$ be a Noetherian local ring,  $M$ a finitely generated $R$-module of dimension $n$ and $I$ an ideal of $R$ such that $\dim(M/IM)=d$. Then $ n-d$ is a lower bound for $c=\operatorname{cd}(I,M)$. Moreover, if $c=n-d$, then 
\begin{equation*}
\operatorname{H}^{d}_{\mathfrak m} (\operatorname{H}^{n-d}_I(M)) \cong \operatorname{H}^{n}_{\mathfrak m}(M)
\end{equation*}
and $\dim \operatorname{Supp}(\operatorname{H}^{n-d}_I(M))= d$.
\end{lemma}
\begin{lemma}\label{l2.4}\emph{( {\cite{Erdogdu-Yildirim}, Theorem 4.4})} Let $(R,\mathfrak m)$ be a Noetherian local ring,  $M$ a finitely generated $R$-module of dimension $n$ and $I$ an ideal of $R$ such that $\dim(M/IM)=1$. Then $\operatorname{H}^{\operatorname{cd}(I,M)}_I(M)$ is Artinian if and only if $\operatorname{cd}(I,M)=n$.
\end{lemma}
Our next result provides a sufficient condition for a top local cohomology module to be non-Artinian:
\begin{lemma} \label{l2.1}Let $(R, \mathfrak m)$ be a Noetherian local ring, $I$ an ideal of $R$ and $M$ an $R$-module(not necessarily finitely generated) with $\operatorname{cd}(I,M)=c$. If there exists an element $x\in \mathfrak m\setminus I$ such that $\operatorname{cd}(I+Rx,M)\neq c$, then $\operatorname{H}^{c}_I(M)$ is not Artinian.
\end{lemma}
\begin{proof} Let $x\in \mathfrak m\setminus I$ such that $\operatorname{cd}(I+Rx,M)\neq c$. Then it follows from the fact $\operatorname{cd}(I+Rx,M)\leq \operatorname{cd}(I,M)+1$ that either $\operatorname{cd}(I+Rx,M)=c+1$ or $\operatorname{cd}(I+Rx,M)<c$. If $\operatorname{cd}(I+Rx,M)=c+1$, then the result follows from Corollary 4.1 of \cite{Dibaei}. Now suppose that $\operatorname{cd}(I+Rx,M)<c$. Then
it follows from the following exact sequence
\begin{equation*}\xymatrix{\cdots \ar[r]& {\underbrace{\operatorname{H}^{c}_{I+Rx}(M)}_{=0}}\ar[r]&{\operatorname{H}^{c}_{I} (M)}\ar[r]& {\operatorname{H}^{c}_{I} (M)}_x \ar[r]& {\underbrace{\operatorname{H}^{c+1}_{I+Rx}(M)}_{=0}} \cdots}
\end{equation*}
that ${\operatorname{H}^{c}_{I} (M)}_x\cong {\operatorname{H}^{c}_{I} (M)}\neq 0$ and so $\dim \operatorname{Supp}({\operatorname{H}^{c}_{I} (M)})\not\subset \{\mathfrak m \}$. Therefore $\operatorname{H}^{c}_I(M)$ is not Artinian.  
 \end{proof}
Let the notations be as in Lemma \ref{l2.3}. The following result  shows that, in many instances,  $\operatorname{H}^{\operatorname{cd}(I,M)}_I(M)$ is non-Artinian in the case when $\operatorname{cd}(I,M)=n-d+1$.
\begin{theorem}\label{t2.2}
Let $(R, \mathfrak m)$ be a Noetherian local ring, $I$ an ideal of $R$ and $M$ a finitely generated $R$-module of dimension $n$ with $\dim(M/IM)=d>1$ and $\operatorname{cd}(I,M)=n-d+1$.  If either
\begin{itemize}
\item[(i)] there exists an element $x\in \mathfrak m\setminus I$ such that $\operatorname{cd}(I+Rx,M)\neq n-d+1$, or
\item[(ii)] $\dim \operatorname{Supp}(\operatorname{H}^{n-d}_I(M))<d$,
\end{itemize}
then $\operatorname{H}^{n-d+1}_I(M)$ is not Artinian.
\end{theorem}
\begin{proof} If $(i)$ holds, then the result follows from Lemma \ref{l2.1}. Now suppose that $\dim \operatorname{Supp}(\operatorname{H}^{n-d}_I(M))<d$ and consider the spectral sequence 
\begin{equation*} 
E^{p,q}_2=\operatorname{H}^p_{\mathfrak m}(\operatorname{H}^q_I(M))  \Longrightarrow \operatorname{H}^{p+q}_{\mathfrak m}(M)
\end{equation*}
and look at the stage $p+q=n$. Since $\operatorname{H}^{n}_{\mathfrak m}(M)\neq 0$, either $\operatorname{H}^d_{\mathfrak m}(\operatorname{H}^{n-d}_I(M))\neq 0$, or $\operatorname{H}^{d-1}_{\mathfrak m}(\operatorname{H}^{n-d+1}_I(M))\neq 0$. Since $\dim \operatorname{Supp}(\operatorname{H}^{n-d}_I(M))<d$, it follows from Grothendieck's vanishing theorem that $\operatorname{H}^d_{\mathfrak m}(\operatorname{H}^{n-d}_I(M))= 0$ and so the spectral sequence degenerates to an isomorphism $\operatorname{H}^{d-1}_{\mathfrak m}(\operatorname{H}^{n-d+1}_I(M))\cong \operatorname{H}^{n}_{\mathfrak m}(M)$. Hence $\operatorname{H}^{n-d+1}_I(M)$ is not Artinian.
\end{proof}

\section{The Main Results}
We begin this section with the following result which provides conditions under which Conjecture \ref{conj1} holds true:
\begin{theorem} \label{t2.4}
Let $(R,\mathfrak m)$ be a Noetherian local unique factorization domain of dimension three and $M$ a finitely generated faithful $R$-module. Then Conjecture  \ref{conj1} holds true.
\end{theorem}

\begin{proof} Let $I$ be an ideal of $R$ with $\operatorname{cd}(I,M)=\dim (M/IM)=2$. Since $I$ is a nonzero ideal and $\operatorname{ht}(I)=\operatorname{ht}I(R/{\operatorname{AnnM}})=\operatorname{ht}_M(I)\leq \dim(M)-\dim(M/IM)=1$, it follows that $\operatorname{ht}(I)=1$ and so $I\subseteq \mathfrak p$ for some height one prime ideal $\mathfrak p$ of $R$. But then since $R$ is a UFD, $\mathfrak p=(x)$ for some $x\in \mathfrak m\setminus I$. Therefore $\operatorname{cd}(I+Rx,M)\leq \operatorname{ara}(I+Rx)=\operatorname{ara}(Rx)=1<2=\operatorname{cd}(I,M)$ and so the result follows from Lemma \ref{l2.1}.
\end{proof}
\begin{corollary} \label{c2.6}
Let $(R,\mathfrak m)$ be a Noetherian local unique factorization domain and $M$ a finitely generated faithful $R$-module of dimension three. Then $\operatorname{H}^{\operatorname{cd}(I,M)}_I(M)$ is Artinian if and only if $\operatorname{cd}(I,M)=3$.
\end{corollary}
\begin{proof}
Let $I$ an ideal of $R$ with $\dim (M/IM)=d$. Knowing that $\operatorname{H}^{3}_I(M)$ is always Artinian [\cite{BS}, Theorem 7.1.6 ], we may, and do, assume that $c<3$ and so $d>0$. Now if $d=1$, then the result follows from Lemma \ref{l2.4}. If, on the other hand, $d=3$, then $I$ is a zero ideal and so $\operatorname{H}^{c}_I(M)= \operatorname{H}^{0}_I(M)=M$ is non-Artinian. Therefore it only remains us to consider the case when $d=2.$ In this case, from Lemma \ref{l2.3}, we have that $n-d=3-2=1$ is a lower bound for $c$ and if, in particular, $c=1$, then $\operatorname{H}^{c}_I(M)$ is non-Artinian. Finally if $c=d=2$, then the result follows from Theorem \ref{t2.4}.
\end{proof}
An important question related to local cohomology theory is to determine when the set of associated prime ideals of a local cohomology module is finite [Fourth problem, \cite{Huneke}]. In many instances, 
the answer to this question is well-known; see eg. \cite{Bahmanpour-Naghipour, Lyu93, Lyu00, Marley}.
Regarding this question, Katzman \cite{Katzman}, Singh \cite{Singh} and later Swanson and Singh \cite{Singh-Swanson} provided some examples of local cohomology modules with infinite set of associated primes. But it still remains an open question that whether the sets of primes that are minimal in the support of local cohomology modules are always finite. This is equivalent to ask whether the support of local cohomology modules of Noetherian rings (or modules) must be Zariski-closed subsets of Spec R. Recently Huneke, Katz and Marley \cite{HKM} and later Khashyarmanesh \cite{Khash} give several positive answers to this question.\\
Our following result shows that if Conjecture \ref{conj1} holds true and the  above open question has an affirmative answer, then all top local cohomology modules of a finitely generated module $M$ are non-Artinian except in the trivial case when $\operatorname{cd}(I,M)= \dim M$.
\begin{theorem} \label{t2.5} Let $(R, \mathfrak m)$ be a Noetherian local ring and $M$ a finitely generated $R$-module of dimension $n$ such that the support of local cohomology modules of $M$ are Zariski-closed subsets of $SpecR$. 
If Conjecture \ref{conj1} holds true, then all top local cohomology modules, $\operatorname{H}^{\operatorname{cd}(I,M)}_I(M)$, are non-Artinian for all ideals $I$ of $R$  with $\operatorname{cd}(I,M)=c<n$.
\end{theorem}

\begin{proof}Let $I$ be an ideal of $R$ with $\dim(M/IM)=d$ and $\operatorname{cd}(I,M)=c<n$. If $c=0$, then it follows from Lemma \ref{l2.3} that $\dim \operatorname{Supp}(\operatorname{H}^0_I(M))=d=n>0$ and so $\operatorname{H}^0_I(M)$ is non-Artinian. Therefore we may, and do, assume that $c>0$ and use induction on $d$ to prove the result.\\
The case $d=1$ follows from Lemma \ref{l2.4}. So suppose, inductively, that $d>1$ and $\operatorname{H}^{\operatorname{cd}(J,M)}_J(M)$  is non-Artinian for all ideals $J$ of $R$ with  $0<\operatorname{cd}(J,M)<n$ and $\dim(M/JM)<d$. \\
We now consider the local cohomology module $\operatorname{H}^{c-1}_I(M)$.\\
\textit{Case 1.} Suppose that $\operatorname{H}^{c-1}_I(M)=0$ and choose an element $x\in \mathfrak m\setminus I$ such that $\dim(M/{(I+Rx)M})=d-1$. If $\operatorname{cd}(I+Rx,M)\neq c$, then the result follows from Lemma \ref{l2.1}. If, on the other hand, $\operatorname{cd}(I+Rx,M)=c$, then, by induction hypothesis, $\operatorname{H}^{c}_{I+Rx}(M)$ is non-Artinian and so from the following exact sequence,
\begin{equation*}\xymatrix{\cdots \ar[r]& {\underbrace{\operatorname{H}^{c-1}_{I}(M_x)}_{=0}}\ar[r]&{\operatorname{H}^{c}_{I+Rx} (M)}\ar[r]& {\operatorname{H}^{c}_{I} (M)} \ar[r]& \cdots}
\end{equation*}
we obtain that ${\operatorname{H}^{c}_{I} (M)}$ is non-Artinian and hence the result follows. \\
\textit{Case 2}. Now suppose that $\operatorname{H}^{c-1}_I(M)\neq 0$. Then from the hypothesis, it follows that the set of minimal prime ideals of $\operatorname{H}^{c-1}_I(M)$, denoted by $\operatorname{Min}_R(\operatorname{H}^{c-1}_I(M))$, is finite.\\
Let now $\operatorname{Min}_R(\operatorname{H}^{c-1}_I(M))=\{\mathfrak q_1,\mathfrak q_2, \cdots, \mathfrak q_m \} $ 
and consider the ideal $K=\cap_{i=1}^m \mathfrak q_i\supseteq I+\operatorname{Ann}M$. It is worth noting that, for any $x\in K$,  $\operatorname {Ass}(\operatorname{H}^{c-1}_{I}(M_x))=\{\mathfrak qR_x: x\notin \mathfrak q\ and\ \mathfrak q\in \operatorname {Ass}(\operatorname{H}^c_I(M))\}= \emptyset$ and since $R$ is Noetherian, $\operatorname{H}^{c-1}_{I}(M_x)=0$. \\
We now need to consider three subcases:
\textit{Case 2(i)}. If there exists an element $x\in K\setminus (I+\operatorname{Ann}M)$ such that\\ $\dim(M/{(I+Rx)M}) < d$, then the result follows by applying the same arguments as the one used in Case 1.\\
\textit{Case 2(ii)}. If on the other hand, $\dim(M/{(I+Rx)M})= d$ for all $x\in K\setminus (I+\operatorname{Ann}M)$, then $K\subseteq \cup_{i=1}^t\{\mathfrak p_i \in \operatorname{Min}_R(I+\operatorname{Ann}M) : \ \dim(R/{\mathfrak {p_i}})=d \}$ and so it follows from prime avoidance lemma that $K=\cap_{i=1}^m \mathfrak q_i\subseteq {\mathfrak {p_i}}$ for some $i$. But then by Proposition 1.11.(ii) of \cite{Atiyah}, we have that $\mathfrak q_i\subseteq {\mathfrak {p_i}}$ and then since $\mathfrak {p_i}$ is a minimal prime ideal of $I+\operatorname{Ann}M$, it follows that $\mathfrak q_i = {\mathfrak {p_i}}$. Therefore $\dim\operatorname{Supp}(\operatorname{H}^{c-1}_I(M))=d$.\\
 Now, it follows from Theorem 2.3 of \cite{Saremi} that $d=\dim\operatorname{Supp}(\operatorname{H}^{c-1}_I(M))\leq n-c+1$ which then implies that $c\leq n-d+1$. But then it follows from Lemma \ref{l2.3} that either $\operatorname{cd}(I,M)=n-d$ or $\operatorname{cd}(I,M)=n-d+1$, and if, moreover, $\operatorname{cd}(I,M)=n-d$, then $\operatorname{H}^{\operatorname{cd}(I,M)}_I(M)$ is non-Artinian. For the remaining case when $\operatorname{cd}(I,M)=n-d+1$, the result follows from our hypothesis.\\
\textit{Case 2(iii)}. Finally, if $K=\cap_{i=1}^m \mathfrak q_i= I+\operatorname{Ann}M$, then $\operatorname{Min}_R(I+\operatorname{Ann}M)\subseteq \operatorname{Min}_R(\operatorname{H}^{c-1}_I(M))$ and so $\operatorname{Supp}(\operatorname{H}^{c-1}_I(M))=V(I+\operatorname{Ann}M)$ which then implies that $d=\dim\operatorname{Supp}(\operatorname{H}^{c-1}_I(M))\leq n-c+1$. The result follows by using the arguments of the previous case.
\end{proof}
We end this paper with the following Corollary which provides conditions under which all top local cohomology modules, but the ones with $\operatorname{cd}(I,M)=\dim M$, are non-Artinian when the underlying ring $R$ is of dimension four.
\begin{corollary} Let $(R,\mathfrak m)$ be a Noetherian local unique factorization domain of dimension four and $M$ a finitely generated faithful $R$-module. Then the following conditions are equivalent:
\begin{itemize}
\item[(i)] All top local cohomology modules, $\operatorname{H}^{\operatorname{cd}(I,M)}_I(M)$, are non-Artinian for all ideals $I$ of $R$ with $\operatorname{cd}(I,M)\neq \dim M$.
\item[(ii)] $\operatorname{H}^{3}_I(M)$ is non-Artinian for every height two ideal $I$ of $R$ with $\dim(M/IM)=2$ and ${\operatorname{cd}(I,M)}=3$.
\end{itemize}
\end{corollary}
\begin{proof} It sufficies to prove that $(ii)$ implies $(i)$. Since $M$ is faithful, we may replace $M$ by $R$. Now since $R$ is of dimension four, it follows from Proposition 3.4 of \cite{HKM} that $\operatorname{Supp}(\operatorname{H}^{i}_I(R))$ is Zariski-closed for all $i$ and for all ideals $I$ of $R$, and therefore by Theorem \ref{t2.5}, it sufficies to show that $\operatorname{H}^{\operatorname{cd}(I,R)}_I(R)$ is non-Artinian for all ideals $I$ of $R$  with $\dim (R/I)=d\geq 2$ and $\operatorname{cd}(I,R)=5-d$. If $d=3$ and $\operatorname{cd}(I,R)=2$, then it follows that $\operatorname{ht}(I)=\dim(R)-\dim(R/I)=1$ and so $I\subseteq \mathfrak p=(x)$ for some height one prime ideal $\mathfrak p$ of $R$. But then $\operatorname{cd}(I+Rx,M)\leq \operatorname{ara}(I+Rx)=\operatorname{ara}(Rx)=1<2=\operatorname{cd}(I,M)$ and so the result follows from Lemma \ref{l2.1}. So the only case left to be considered is when $d=2$ and $\operatorname{cd}(I,R)=3$. In this case, $\operatorname{ht}(I)\leq \dim(R)-\dim(R/I)=2$ and, from the hypothesis, if $\operatorname{ht}(I)=2$, then $\operatorname{H}^{3}_I(M)$ is non-Artinian. If, on the other hand, $\operatorname{ht}(I)=1$,  then use the same arguments as above and see that $\operatorname{H}^{3}_I(M)$ is non-Artinian. 
\end{proof}
\begin{ack}
The author would like to thank Professor Vahap Erdo\u{g}du for his careful reading of the first draft of this manuscript and many helpful suggestions.
\end{ack}
\footnotesize{\bibliographystyle{plainnat}
}
\end{document}